\documentclass[12pt, reqno, twoside, letterpaper]{amsart}

\usepackage{amsmath,amssymb,amsbsy,amsfonts,amsthm,latexsym,
	amsopn,amstext,amsxtra,euscript,amscd,stmaryrd,mathrsfs,cite,array}

\numberwithin{equation}{section}

\usepackage{color}
\usepackage[usenames,dvipsnames,svgnames,table]{xcolor}

\def\eps{\varepsilon}

\def\fl#1{\left\lfloor#1\right\rfloor}

\def\({\left(}
\def\){\right)}

\newcommand{\e}{\ensuremath{\mathbf{e}}}

\newcommand{\cL}{\ensuremath{\mathcal{L}}}

\newcommand{\cN}{\ensuremath{\mathcal{N}}}

\newcommand{\cS}{\ensuremath{\mathcal{S}}}

\newcommand{\fA}{\ensuremath{\mathfrak{A}}}

\newcommand{\fT}{\ensuremath{\mathfrak{T}}}

\newcommand{\RR}{\ensuremath{\mathbb{R}}}

\usepackage[margin=2.5cm]{geometry}

\usepackage{amsthm}
\newtheoremstyle{customthm}
{1em}                    
{1em}                    
{\itshape}               
{}                       
{\scshape}               
{.}                      
{5pt plus 1pt minus 1pt} 
{}                       

\newtheoremstyle{customrem}
{1em}                    
{1em}                    
{}                       
{}                       
{\scshape}               
{.}                      
{5pt plus 1pt minus 1pt} 
{}                       

\theoremstyle{customthm}

\newtheorem{X}{X}[section]
\newtheorem{theorem}[X]{Theorem}
  
\newtheorem{lemma}[X]{Lemma}

\theoremstyle{customrem}

\usepackage{etoolbox}
\AtEndEnvironment{remark}{\null\hfill\qedsymbol}

\AtEndEnvironment{definition}{\null\hfill\qedsymbol}

\renewcommand{\le}{\ensuremath{\leqslant}}
\renewcommand{\ge}{\ensuremath{\geqslant}}
\def\fl#1{\left\lfloor#1\right\rfloor}

\renewcommand{\pod}[1]{\mathchoice
	{\allowbreak \if@display \mkern 5mu\else \mkern 5mu\fi (#1)}
	{\allowbreak \if@display \mkern 5mu\else \mkern 5mu\fi (#1)}
	{\mkern4mu(#1)}
	{\mkern4mu(#1)}
}

\newcommand*{\defeq}{\mathrel{\vcenter{\baselineskip0.5ex \lineskiplimit0pt
			\hbox{\scriptsize.}\hbox{\scriptsize.}}}%
	=}

\DeclareSymbolFont{EUEX}{U}{euex}{m}{n}

\DeclareSymbolFont{euexlargesymbols}{U}{euex}{m}{n}
\DeclareMathSymbol{\intop}{\mathop}{euexlargesymbols}{"52}
\def\int{\intop\nolimits}

\DeclareSymbolFont{euexsymbols}     {U}{euex}{m}{n}
\DeclareMathSymbol{\smallint}{\mathop}{euexsymbols}{"52}

\usepackage{ifthen}
\usepackage{xifthen}

\def\sums{                        
	\@ifnextchar[
	{\sums@i}
	{\ensuremath{\sum}}    
}
\def\sums@i[#1]{
	\@ifnextchar[
	{\sums@ii{#1}}
	{\ensuremath{\sum_{#1}}}
}
\def\sums@ii#1[#2]{
	\@ifnextchar[
	{\sums@iii{#1}{#2}}
	{\ensuremath{\sum_{\substack{#1 \\ #2}}}}
}
\def\sums@iii#1#2[#3]{
	\@ifnextchar[
	{\sums@iv{#1}{#2}{#3}}
	{\ensuremath{\sum_{\substack{#1 \\ #2 \\ #3}}}}
}
\def\sums@iv#1#2#3[#4]{
	\@ifnextchar[
	{\sums@v{#1}{#2}{#3}{#4}}
	{\ensuremath{\sum_{\substack{#1 \\ #2 \\ #3 \\ #4}}}}
}
\def\sums@v#1#2#3#4[#5]{
	{\ensuremath{\sum_{\substack{#1 \\ #2 \\ #3 \\ #4 \\ #5}}}}
}

\def\sumss[#1]{
	\@ifnextchar[
	{\sumss@i[#1]}
	{
		\ifthenelse{\isempty{#1}} 
		{\ensuremath{\sum}}       
		{
			\ifthenelse{\equal{#1}{'}}
			{\ensuremath{\sideset{}{^{\prime}}{\sum}}}
			{\ensuremath{\sideset{}{^{#1}}{\sum}}} 
		}  
	}
}    
\def\sumss@i[#1][#2]{
	\@ifnextchar[
	{\sumss@ii[#1]{#2}}
	{
		\ifthenelse{\isempty{#1}} 
		{\ensuremath{\sum_{#2}}}       
		{
			\ifthenelse{\equal{#1}{'}}
			{\ensuremath{\sideset{}{^{\prime}}{\sum}_{#2}}}
			{\ensuremath{\sideset{}{^{#1}}{\sum}_{#2}}} 
		}  
	}
}
\def\sumss@ii[#1]#2[#3]{
	\@ifnextchar[
	{\sumss@iii[#1]{#2}{#3}}
	{
		\ifthenelse{\isempty{#1}} 
		{\ensuremath{\sum_{\substack{#2 \\ #3}}}}       
		{
			\ifthenelse{\equal{#1}{'}}
			{\ensuremath{\sideset{}{^{\prime}}{\sum}_{\substack{#2 \\ #3}}}}
			{\ensuremath{\sideset{}{^{#1}}{\sum}_{\substack{#2 \\ #3}}}}
		} 
	}
}
\def\sumss@iii[#1]#2#3[#4]{
	\@ifnextchar[
	{\sumss@iv[#1]{#2}{#3}{#4}}
	{
		\ifthenelse{\isempty{#1}} 
		{\ensuremath{\sum_{\substack{#2 \\ #3 \\ #4}}}}       
		{
			\ifthenelse{\equal{#1}{'}}
			{\ensuremath{\sideset{}{^{\prime}}{\sum}_{\substack{#2 \\ #3 \\ #4}}}}
			{\ensuremath{\sideset{}{^{#1}}{\sum}_{\substack{#2 \\ #3 \\ #4}}}} 
		}  
	}
}
\def\sumss@iv[#1]#2#3#4[#5]{
	\@ifnextchar[
	{\sumss@v[#1]{#2}{#3}{#4}{#5}}
	{
		\ifthenelse{\isempty{#1}} 
		{\ensuremath{\sum_{\substack{#2 \\ #3 \\ #4 \\ #5}}}}       
		{
			\ifthenelse{\equal{#1}{'}}
			{\ensuremath{\sideset{}{^{\prime}}{\sum}_{\substack{#2 \\ #3 \\ #4 \\ #5}}}}
			{\ensuremath{\sideset{}{^{#1}}{\sum}_{\substack{#2 \\ #3 \\ #4 \\ #5}}}} 
		}  
	}
}
\def\sumss@v[#1]#2#3#4#5[#6]{
	{\ifthenelse{\isempty{#1}} 
		{\ensuremath{\sum_{\substack{#2 \\ #3 \\ #4 \\ #5 \\ #6 }}}}       
		{
			\ifthenelse{\equal{#1}{'}}
			{\ensuremath{\sideset{}{^{\prime}}{\sum}_{\substack{#2 \\ #3 \\ #4 \\ #5 \\ #6 }}}}
			{\ensuremath{\sideset{}{^{#1}}{\sum}_{\substack{#2 \\ #3 \\ #4 \\ #5 \\ #6 }}}} 
		}  
	}
}

\def\sumsstxt[#1]{
	\@ifnextchar[
	{\sumsstxt@i[#1]}
	{
		\ifthenelse{\isempty{#1}} 
		{\ensuremath{\textstyle\sum}}       
		{
			\ifthenelse{\equal{#1}{'}}
			{\ensuremath{\sideset{}{^{\prime}}{\textstyle\sum}}}
			{\ensuremath{\sideset{}{^{#1}}{\textstyle\sum}}} 
		}  
	}
}    
\def\sumsstxt@i[#1][#2]{
	\@ifnextchar[
	{\sumsstxt@ii[#1]{#2}}
	{
		\ifthenelse{\isempty{#1}} 
		{\ensuremath{\textstyle\sum_{#2}}}       
		{
			\ifthenelse{\equal{#1}{'}}
			{\ensuremath{\sideset{}{^{\prime}}{\textstyle\sum}_{#2}}}
			{\ensuremath{\sideset{}{^{#1}}{\textstyle\sum}_{#2}}} 
		}  
	}
}
\def\sumsstxt@ii[#1]#2[#3]{
	\@ifnextchar[
	{\sumsstxt@iii[#1]{#2}{#3}}
	{
		\ifthenelse{\isempty{#1}} 
		{\ensuremath{\textstyle\sum_{\substack{#2 \\ #3}}}}       
		{
			\ifthenelse{\equal{#1}{'}}
			{\ensuremath{\sideset{}{^{\prime}}{\textstyle\sum}_{\substack{#2 \\ #3}}}}
			{\ensuremath{\sideset{}{^{#1}}{\textstyle\sum}_{\substack{#2 \\ #3}}}}
		} 
	}
}
\def\sumsstxt@iii[#1]#2#3[#4]{
	\@ifnextchar[
	{\sumsstxt@iv[#1]{#2}{#3}{#4}}
	{
		\ifthenelse{\isempty{#1}} 
		{\ensuremath{\textstyle\sum_{\substack{#2 \\ #3 \\ #4}}}}       
		{
			\ifthenelse{\equal{#1}{'}}
			{\ensuremath{\sideset{}{^{\prime}}{\textstyle\sum}_{\substack{#2 \\ #3 \\ #4}}}}
			{\ensuremath{\sideset{}{^{#1}}{\textstyle\sum}_{\substack{#2 \\ #3 \\ #4}}}} 
		}  
	}
}
\def\sumsstxt@iv[#1]#2#3#4[#5]{
	{\ifthenelse{\isempty{#1}} 
		{\ensuremath{\textstyle\sum_{\substack{#2 \\ #3 \\ #4 \\ #5}}}}       
		{
			\ifthenelse{\equal{#1}{'}}
			{\ensuremath{\sideset{}{^{\prime}}{\textstyle\sum}_{\substack{#2 \\ #3 \\ #4 \\ #5}}}}
			{\ensuremath{\sideset{}{^{#1}}{\textstyle\sum}_{\substack{#2 \\ #3 \\ #4 \\ #5}}}} 
		}  
	}
}

\def\prods{              
	\@ifnextchar[
	{\prods@i}
	{\ensuremath{\prod}}    
}
\def\prods@i[#1]{
	\@ifnextchar[
	{\prods@ii{#1}}
	{\ensuremath{\prod_{#1}}}
}
\def\prods@ii#1[#2]{
	\@ifnextchar[
	{\prods@iii{#1}{#2}}
	{\ensuremath{\prod_{\substack{#1 \\ #2}}}}
}
\def\prods@iii#1#2[#3]{
	\@ifnextchar[
	{\prods@iv{#1}{#2}{#3}}
	{\ensuremath{\prod_{\substack{#1 \\ #2 \\ #3}}}}
}
\def\prods@iv#1#2#3[#4]{
	{\ensuremath{\prod_{\substack{#1 \\ #2 \\ #3 \\ #4}}}}
}
\newcommand{\RNum}[1]{\uppercase\expandafter{\romannumeral #1\relax}}
\allowdisplaybreaks

\title[The Piatetski-Shapiro prime number theorem]
{The Piatetski-Shapiro prime number theorem}

\author[Lingyu Guo]{Lingyu Guo}

\address{School of Mathematics and Statistics, Xi'an Jiaotong University, Xi'an,Shaanxi,China.}

\email{guo.lingyu@foxmail.com}

\author[Victor Zhenyu Guo]{Victor Zhenyu Guo}

\address{School of Mathematics and Statistics, Xi'an Jiaotong University, Xi'an, Shaanxi, China.}

\email{vzguo@foxmail.com; guozyv@xjtu.edu.cn}

\author[Li Lu]{Li Lu}

\address{School of Mathematics and Statistics, Xi'an Jiaotong University, Xi'an,Shaanxi,China.}

\email{lilu\_math@foxmail.com}

\date{\today}

\begin{document}
	
\begin{abstract}
The Piatetski-Shapiro sequences are of the form $\mathcal{N}_{c} := (\lfloor n^{c} \rfloor)_{n=1}^\infty$, where $\lfloor \cdot \rfloor$ is the integer part. It is expected that there are infinitely many primes in a Piatetski-Shapiro sequence for $c \in (1,2)$. In this article, we prove there are infinitely many Piatetski-Shapiro prime numbers for $1 < c < 1.1612\dots$ with an asymptotic formula. As a key idea, we prove a new bound for related type $I$ sum. 
\end{abstract}
	
\maketitle
	
\begin{quote}
\end{quote}

\begin{quote}
\end{quote}
	
\newcommand{\tind}[1]{\ensuremath{\widetilde{\mathbf{1}}_{#1}}} 
	

\section{Introduction}
	
The Piatetski-Shapiro sequences are sequences of the form
$$
\cN_c \defeq (\lfloor n^{c} \rfloor)_{n=1}^\infty,
$$
where $\fl{\cdot}$ is the integer part. Piatetski-Shapiro~\cite{PS} proved the Piatetski-Shapiro prime number theorem stating that for $ 1 < c < 12/11$ the counting function
$$
\pi_c(x) \defeq \# \big\{\text{\rm prime~} p\le x : p \in \cN_c \big\}
$$
satisfies the asymptotic relation
\begin{equation}
	\label{eq:PS}
	\pi_c(x) = (1 + o(1)) \frac{x^{1/c}}{\log x} \qquad \text{~\rm as } x \to \infty.
\end{equation}
The asymptotic relation is expected to hold for all values of $1 < c < 2$ and was proved by Leitmann and Wolke \cite{LeWo} for almost all $c \in (1,2)$ in the sense of Lebesgue measure. The estimation of Piatetski-Shapiro primes is an approximation of the well-known conjecture that there exist infinitely many primes of the form $n^2+1$. 
	
The upper bound of the admissible range for $c$ of the above formula has been extended for many times. We provide a table with all the previous results; see Table~\ref{tab:psa}.
	
\begin{table}[h!]
	\caption{The range of $c$ when equation~\eqref{eq:PS} holds}  
	\label{tab:psa}
	\begin{tabular}{||c|c||}
		\hline\hline
		\vphantom{\Big|} Authors  &  The upper bound of $c$ \\ 
		\hline 
	    Piatetski-Shapiro~\cite{PS}& $\frac{12}{11} = 1.0909\dots$  \\
		Kolesnik~\cite{Kole}  & $\frac{10}{9} = 1.1111\dots$  \\
		Graham; Leitmann~\cite{Leit} &  $\frac{69}{62} = 1.1129\dots$   \\
		Heath-Brown~\cite{HB2} & $\frac{755}{662} = 1.1404\dots$ \\
		Kolesnik~\cite{Kole2}  & $\frac{39}{34} = 1.1470\dots$ \\
		Liu and Rivat~\cite{LiRi} & $\frac{15}{13} = 1.1538\dots$ \\
		Rivat~\cite{Riva}  & $\frac{6121}{5302} = 1.1544\dots$ \\
		Rivat and Sargos~\cite{RiSa} & $\frac{2817}{2426} = 1.1611\dots$ \\
		\hline\hline
	\end{tabular}
\end{table}
	
Moreover, to prove there are infinitely many Piatetski-Shapiro primes, Rivat considered to prove a lower bound of the counting function
\begin{equation}\label{eq:PSl}
	\pi_c(x) \gg \frac{x^{1/c}}{\log x}.
\end{equation}
We also list all the results in this consideration (Table~\ref{tab:psl}). 
	
\begin{table}[h!]
	\caption{The range of $c$ when equation~\eqref{eq:PSl} holds}  
	\label{tab:psl}
	\begin{tabular}{||c|c||}
		\hline\hline
		\vphantom{\Big|} Authors  &  The upper bound of $c$ \\ 
		\hline 
		Rivat~\cite{Riva}  & $\frac{7}{6} = 1.1666\cdots$ \\
		Baker, Harman and Rivat~\cite{BaHa}; Jia~\cite{Jia3}  & $ \frac{20}{17} = 1.1764\dots$  \\
		Jia~\cite{Jia2}  & $\frac{13}{11} = 1.1818\dots$  \\
		Kumchev~\cite{Kumc} & $\frac{45}{38} = 1.1842\dots$ \\ 
		Rivat and Wu~\cite{RiWu} & $ \frac{243}{205} = 1.1853\dots$ \\
		\hline\hline
	\end{tabular}
\end{table}

In this article, we extend the admissible range of $c$. 
	
\begin{theorem}\label{thm:main}
	For 
	$$
	1 < c < \frac{10318869}{8886224} = 1.1612\dots,
	$$
	we have that 
	$$
	\pi_c(x) = \frac{x^{1/c}}{\log x} + O\bigg(\frac{x^{1/c}}{\log^2 x} \bigg).
	$$ 
\end{theorem}

Our proof is highly related to the following exponential sum 
$$
\fT \defeq \sum_{h \sim H} \sum_{m \sim M} \sum_{n \sim N} a_h b_m \e \bigg( X\frac{h^\alpha m^\beta n^\gamma}{H^{\alpha} M^{\beta} N^{\gamma} } \bigg)
$$
with $X > 0$, $H, M, N \ge 1$ and $|a_h|, |b_m| \le 1$.  We prove a refinement of the bound by Wu \cite{Wu2002}. 

\begin{theorem}
	\label{thm:2}
	Let $(\kappa,\lambda)$ be an exponent pair, $\alpha, \beta, \gamma \in \mathbb{R}$ with $\alpha\beta\gamma(1-\alpha) \neq 0$ and $(\gamma-1)/(1-\alpha) \notin \mathbb{N}$, $\cL \defeq \log(2 + XHMN)$. It follows that
	\begin{align*}
		\fT \cL^{-1} 
		&\ll X^{\frac{1+2\kappa}{2(2 + \kappa + \lambda)}} H^{\frac{\kappa + \lambda + 1}{2+\kappa + \lambda}} M^{\frac{\kappa+\lambda + 4}{2(2+\kappa + \lambda)}} N^{\frac{2-\kappa+3\lambda}{2(2+\kappa+\lambda)}} + X^{\frac{3\kappa + \lambda + 1}{4(1+\kappa + \lambda)}} H^{\frac{1}{2}} M N^{\frac{1 + \lambda - \kappa}{2(1 + \kappa + \lambda)}} \\
		&\qquad + X^{\frac{\kappa-\lambda +1}{4(1+\kappa+\lambda)}} H^{\frac{1}{2}} M N^{\frac{\kappa + 3\lambda + 1}{2(1+\kappa + \lambda)}} + X^{\frac{5\kappa+\lambda + 2}{4(2+\kappa+\lambda)}} H^{\frac{1}{2}} M^{\frac{3\kappa+3\lambda + 8}{4(2+\kappa+\lambda)}} N^{\frac{4+5\lambda-3\kappa}{4(2+\kappa+\lambda)}} \\
		&\qquad + X^{\frac{1}{4}}H^{\frac{1}{2}}M^{\frac{13}{12}} N^{\frac{1}{12}} + H M^{\frac{2}{3}} N^{\frac{2}{3}} + X^{-\frac{1}{4}}H^{\frac{1}{2}} M^{\frac{13}{12}} N^{\frac{13}{12}} \\
		&\qquad + X^{\frac{1}{4}}H^{\frac{1}{2}} M^{\frac{11}{12}} N^{\frac{5}{12}} + X^{\frac{1+2\kappa}{4}} H^{\frac{1}{2}} M^{\frac{4-\kappa-\lambda}{4}} N^{\frac{2+\lambda-3\kappa}{4}} \\
		&\qquad + X^{1/4}H^{\frac{1}{2}} M^{\frac{1}{2}} N + H^{\frac{1}{2}} M N + X^{-\frac{1}{2}}HMN.
	\end{align*}
\end{theorem}

Our theorem~\ref{thm:2} gives a better bound than Wu \cite{Wu2002} in most cases, especially when the terms with exponent pairs are the worst terms. We also remove the restriction of $X$. For a detailed comparison to Wu's result, we refer the readers to Section \ref{sec:com}. 

\section{Preliminaries}
	
\subsection{Notation}
	
We denote by $\fl{t}$ and $\{t\}$ the integer part and the fractional part of $t$, respectively. We define $||t|| \defeq \min\{ |t-n|: n \in \mathbb{Z}\}$. As is customary, we put $\e(t)\defeq e^{2\pi it}$. We make considerable use of the sawtooth function defined by
$$
\psi(t) \defeq t-\fl{t}-\frac{1}{2}=\{t\}-\frac{1}{2}\qquad(t\in\RR).
$$
The letter $p$ always denotes a prime. For the Piatetski-Shapiro sequence $(\fl{n^c})_{n=1}^\infty$, we denote $\gamma \defeq c^{-1}$. We use notation of the form $m\sim M$ as an abbreviation for $M< m\le 2M$. $\eps$ is always a sufficiently small positive number.
	
Throughout the paper, implied constants in symbols $O$, $\ll$ and $\gg$ may depend (where obvious) on the parameters $c, \eps$ but are absolute otherwise. For given functions $F$ and $G$, the notations $F\ll G$, $G\gg F$ and $F=O(G)$ are all equivalent to the statement that the inequality $|F|\le C|G|$ holds with some constant $C>0$. $F \asymp G$ means that $F \ll G \ll F$. 
	
\subsection{Technical lemmas}
	
We need the following well known approximation of Vaaler.
	
\begin{lemma}
	\label{lem:Vaaler}
	For any $H\ge 1$ there are numbers $a_h,b_h$ such that
	$$
	\bigg|\psi(t)-\sum_{0<|h|\le H}a_h\,\e(th)\bigg| \le\sum_{|h|\le H}b_h\,\e(th),\qquad a_h\ll\frac{1}{|h|}\,,\qquad b_h\ll\frac{1}{H}\,.
	$$
\end{lemma}
\begin{proof}
	See~\cite{Vaal}.
\end{proof}

\begin{lemma}[Kusmin-Laudau]
	\label{lem:KL}
	If $f$ is continuously differentiable, $f'$ is monotonic and $||f'|| \ge \lambda > 0$ on an interval $I$ then
	$$
	\sum_{n \in I} \e(f(n)) \ll \lambda^{-1}. 
	$$
\end{lemma}
\begin{proof}
	See Theorem 2.1 in \cite{GraKol}.
\end{proof}

\begin{lemma}
	\label{lem:A}
	Let $N,Q \ge 1$, $z_n \in \mathbb{C}$ and $x_n \in \RR$. If 
	$$
	\max_{1 \le j, k \le N} |x_j-x_k| \le 1 - \frac{1}{Q},
	$$
	then
	$$
	\big| \sum_{n \le N} z_n \big|^2 \le Q \mathop{\sum_{j\le N}\sum_{k \le N}}_{|x_j-x_k| \le 1 - 1/Q} (1-Q|x_j-x_k|) z_j \bar{z_k}. 
	$$
\end{lemma}
\begin{proof}
See Lemma 2.1 in \cite{Wu2002}.
\end{proof}
	
The following lemma is derived from the Poisson summation formula, known as the B-process; see \cite[Lemma 3.6]{GraKol}.
	
\begin{lemma}
	\label{lem:B}
	Suppose that $f$ has four continuous derivatives on $[a,b]$ and that $f''<0$ on this interval. Suppose further that $[a,b] \subset [N, 2N]$ and that $\alpha = f'(b)$ and $\beta = f'(a)$. Assume that there is some $F>0$ such that 
	$$
	f^{(2)}(x) \asymp FN^{-2},~f^{(3)}(x) \ll FN^{-3}, \text{~and~} f^{(4)}(x) \ll FN^{-4} 
	$$
	for $x$ in $[a,b]$. Let $x_{\nu}$ be defined by the relation $f'(x_\nu) = \nu$, and let $\phi(\nu) = -f(x_\nu) + \nu x_\nu$. Then
	$$
	\sum_{n \in I} \e(f(n)) = \sum_{\alpha \le \nu \le \beta} \frac{\e(-\phi(\nu) - 1/8)}{|f''(x_\nu)|^{1/2}} + O(\log(FN^{-1} + 2) + F^{-1/2} N). 
	$$
\end{lemma}

\begin{lemma} \label{lem:7}
	Suppose that
	\begin{align*}
		L(H)=\sum_{i=1}^{m} A_{i} H^{a_{i}}+\sum_{j=1}^{n} B_{j} H^{-b_{j}},
	\end{align*}
	where $A_{i}, B_{j}, a_{i}$ and $b_{j}$ are positive. Assume further that $H_{1} \le H_{2}$. Then there exists some $\mathcal{H}$ with $H_{1} \le \mathcal{H} \le H_{2}$ and
	\begin{align*}
		L(\mathcal{H}) \ll \sum_{i=1}^{m} A_{i} H_{1}^{a_{i}}+\sum_{j=1}^{n} B_{j} H_{2}^{-b_{j}}+\sum_{i=1}^{m} \sum_{j=1}^{n}\big(A_{i}^{b_{j}} B_{j}^{a_{i}}\big)^{\frac{1}{a_{i}+b_{j}}} .
	\end{align*}
	The implied constant depends only on $m$ and $n$.
\end{lemma}
\begin{proof}
	\rm See Lemma 3 of Srinivasan \cite{Srinivasan1963}.
\end{proof}

The following spacing lemma is useful for counting binomial points $m^\alpha n^\beta$. 

\begin{lemma}
	\label{lem:spacing}
	Let $\alpha\beta \neq 0, \Delta > 0, M \ge 1$ and $N \ge 1$. Let $\fA(M, N; \Delta)$ be the number of quadruples $(m, \tilde{m}, n, \tilde{n})$ such that
	$$
	\bigg|\bigg(\frac{\tilde{m}}{m}\bigg)^\alpha - \bigg(\frac{\tilde{n}}{n}\bigg)^\beta \bigg| < \Delta,
	$$
	with $m, \tilde{m} \sim M$ and $n, \tilde{n} \sim N$. We have
	$$
	\fA(M, N; \Delta) \ll MN \log 2MN + \Delta M^2 N^2. 
	$$
\end{lemma}    
\begin{proof}
	See Lemma 1 in \cite{FI1989}.
\end{proof}
	
We provide the method of exponent pair by the following lemma.
	
\begin{lemma}\label{lem: exponent pair method}
	Let $s\geqslant 2$ be a positive integer, and let $f(x)$ be a real valued function with $s$ continuous derivatives on $[N,2N]$ such that
	$$
	|f^{(s)}(n)|\asymp YN^{1-s}.
	$$
	Then we have
	$$
	\sum_{n\sim N}\mathbf{e}(f(n))\ll Y^{\kappa}N^{\lambda}+Y^{-1},
	$$
	where $(\kappa,\lambda)$ is any exponent pair.
\end{lemma}

\begin{proof}
	See \cite[Chapter 3]{GraKol} or \cite[Lemma 1]{CaoZ1998}.
\end{proof}

The following lemma is useful for Type $I$ bound. 

\begin{lemma}
	\label{lem:ty}
	Let $H, N, M$ be positive integers, $X$ be a real number greater than $1$, $a_{h,n}$ and $b_m$ be complex numbers of modulus at most $1$, $\alpha, \beta, \gamma$ be fixed real numbers such that $\alpha (\alpha -1) \beta \neq 0$, then for the one-dimensional exponential sum, it follows that
	\begin{align*}
		&\sum_{h \sim H}  \sum_{n \sim N} \max_{M \le M_1 \le M_2 \le 2M} \bigg| \sum_{M_1 \le m \le M_2} \e \bigg(X\frac{m^\alpha h^\beta n^{\gamma}}{M^\alpha H^\beta N^\gamma} \bigg)\bigg| \\
		&\qquad \qquad \ll (HNM)^{1+\eps} \bigg\{ \bigg(\frac{X}{HNM^2}\bigg)^{1/4}  + \frac{1}{M^{1/2}} + \frac{1}{X} \bigg\}.
	\end{align*}
\end{lemma}
\begin{proof} 
	See Theorem 3 in \cite{RS2006}.
\end{proof}
	
\section{Ideas of the proof}\label{sec:3}
	
The key point to prove the Piatetski-Shapiro prime number theorem is to estimate the admissible range of $\gamma$ of the relation
$$
\cS \defeq \sum_{h \sim H} \delta_h \sum_{n \sim x} \Lambda(n) \e(hn^\gamma) \ll x^{1-\eps}
$$
where $\gamma \defeq 1/c$. By a decomposition of von Mangoldt function, we mainly need to bound the sum
\begin{equation}
\label{eq:S}
\sum_{h\sim H} \delta_h \sum_{m\sim M}\sum_{n\sim N}a_{m}b_{n}\e(h(mn)^{\gamma}),
\end{equation}
where $0<\gamma<1$, $MN\asymp x$, $H\leqslant x^{1-\gamma+\eps}$ and $\delta_h$, $a_m$, $b_n$ are complex numbers of modulus at most $1$. If the coefficients $a_m$ and $b_n$ satisfy the conditions
\begin{equation}\label{eq:SI}
	|a_m| \ll 1,\quad b_n = 1 ~\text{or}~\log n,
\end{equation}
we denote the sum by $S_I$. If they satisfy the conditions
\begin{equation}\label{eq:SII}
	|a_m| \ll 1,\quad b_n \ll 1 ,
\end{equation}
we denote the sum by $S_{II}$. The estimation of type $II$ sum is usually the barrier to improve the range of $c$.

\subsection{Classical results}
Heath-Brown \cite{HB2} proved a remarkable type $II$ bound. 
\begin{lemma}
\label{lem:HBTII}
For 
$$
x^{1-\gamma + \eps} \ll N \ll x^{5\gamma-4-\eps},
$$
we have $S_{II} \ll x^{1-\eps}$.
\end{lemma}
Moreover, Heath-Brown also provided a new decomposition of von Mangoldt function, nowadays known as the Heath-Brown's identity, which is the following lemma. 

\begin{lemma}
	\label{lem:HB2}
	Let $3\leqslant U<V<Z<x$ and suppose that $Z-\frac{1}{2}\in \mathbb{N}$, $x\geqslant 64Z^{2}U,~Z\geqslant4U^{2},~V^{3}\geqslant 32x$. Assume further that $f(n) = 0$ when $n\leqslant x$ or $n > 2x$ and that $|f(n)| \leqslant f_{0}$ otherwise. Then the sum
	$$
	\sum_{n \sim x}\Lambda(n)f(n)
	$$
	may be decomposed into $O(\log^{10}x)$ sums, each either of Type $I$ with $N>Z$, or of Type $II$ with $U<N<V$.
\end{lemma}

Liu and Rivat \cite{LiRi} applied the double large sieve method to type $I$ sums together with the Heath-Brown's identity, proving the Piatetski-Shapiro prime number theorem for $c$ up to $15/13$. This reaches the limitation of the type $II$ bound since by Lemma \ref{lem:HB2} it is easy to see that we requires $V \gg x^{1/3}$. The upper bound of Heath-Brown's type $II$ bound gives that $5\gamma-4 > 1/3$ implies that $\gamma > 13/15$. 

To break the type $II$ bound barrier, Rivat and Sargos \cite{RiSa} applied the same idea of Heath-Brown's identity but gave a more suitable decomposition of $\Lambda(n)$ for Piatetski-Shapiro prime numbers. 

\begin{lemma}
\label{lem:RS}
Let $l \ge 4$ and $1/2l \le \alpha \le 1/6$. With the above notations, if
\begin{enumerate}
	\item[(i)] type $II$ sum $\displaystyle \sum_{m \sim M} \sum_{n \sim N} a_m b_n f(mn) \ll U$ for $x^\alpha \le N \le x^{2\alpha}$;
	\item[(ii)] type $I'$ sum $\displaystyle \sum_{m \sim M} \sum_{n \sim N} a_m f(mn) \ll U$ for $x^{2\alpha} \le N \le x^{1/3}$;
	\item[(iii)] type $I$ sum $\displaystyle \sum_{m \sim M} \sum_{n \sim N} a_m f(mn) \ll U$ for $x^{(1-\alpha)/2} \le N$;
\end{enumerate}
then $\displaystyle \sum_{n \sim x} \Lambda(n)f(n) \ll U x^\eps$. 
\end{lemma}

The new decomposition allows the range of $\gamma$ to be $2(1-\gamma) < 5\gamma - 4$, which is $\gamma > 6/7$ from the type $II$ bound. Therefore, it is sufficient to prove a better type $I'$ bound to improve the admissible range of Piatetski-Shapiro prime number theorem. With a new idea for type $I'$ sum, Rivat and Sargos \cite{RiSa} proved the best range of $c$ at the time. We prove a better admissible range of $c$ because our Theorem \ref{thm:2} gives a better estimation of type $I'$ sum. 

\subsection{A comparison to Wu's bound} \label{sec:com}
Wu \cite{Wu2002} also considered this type $I'$ sum. He applied a similar idea of Rivat and Sargos \cite{RiSa}, proving a more general theorem, not only for the exponential sums related to Piatetski-Shapiro primes, but also for more general choices of $H, M, N, X$. To make a comparison between Wu's result and our Theorem $\ref{thm:2}$, we list Wu's bound as the following. Let $k \ge 2$ and $K \defeq 2^k$. If $X \le \min\{H^2, H^2 M^{-1} N \}$, then
\begin{align}
	\label{eq:Wu}
	\fT \cL^{-1} &\ll X^{\frac{K}{6K-2k-8}} H^{\frac{4K - 2k -4}{6K-2k-8}} M^{\frac{5K - k - 8}{6K-2k-8}} N^{\frac{5K-3k-8}{6K-2k-8}} + X^{\frac{1}{4}} H^{\frac{1}{2}} M^{\frac{1}{2}} N \nonumber \\
	&\qquad + X^{\frac{1}{4}} H^{\frac{1}{2}} M N^{\frac{1}{2}} + HM + HM^{\frac{1}{2}} N^{\frac{1}{2}} + H^{\frac{1}{2}}MN + X^{-\frac{1}{2}} H M N. 
\end{align}

Our theorem \ref{thm:2} refines \eqref{eq:Wu}. The idea of our proof is similar to Wu's idea, but with a few important differences. A key idea is that we employ exponent pair method to the proof instead of higher derivative tests. A first glance shows that we replace the first term of $\eqref{eq:Wu}$ into a term with exponent pairs, but the secondary term of exponent pair bound also makes the estimations more complicate. 

Another difference is that Wu applied a generalized B-process to \eqref{eq:afC} while we apply a normal B-process. As a consequence, Wu had a easier main term but much complicated error term, but we have simpler error terms. The readers may compare our equation $\eqref{eq:T}$ with equation (2.8) in \cite{Wu2002}. It may happen that our error term is not as good as Wu when $X$ is very small. However, since the proof assumes $X \ge MN$, this may rarely happen. 

We mention that Wu ignored details in his proofs. For example when optimizing $Q$ at the end of his proof, the range of $Q$ should be $[Q_0, MN]$ instead of $[Q_0, \infty)$. We completely believe that Wu's result will not be affected by this calculation, but a few more negligible terms should be added as a rigorous proof. Moreover, we also remove the restriction of $X$ for a more general theorem. 
	
\section{Proof of Theorem 2}
	
By the Cauchy-Schwarz inequality we have
$$
|\fT|^2 \ll H \sum_{h \sim H} \bigg| \sum_{m \sim M} \sum_{n \sim N} b_m \e\bigg(X \frac{h^\alpha m^\beta n^\gamma}{H^\alpha M^\beta N^\gamma}\bigg) \bigg|^2.  
$$
Now let $Q \ge 10$. Applying Lemma \ref{lem:A} we obtain that 
\begin{equation}
\label{eq:afC}
|\fT|^2 \ll H Q \sum_{m \sim M} \sum_{\tilde{m} \sim M} \bigg| \mathop{\sum_{n \sim N} \sum_{\tilde{n} \sim N}}_{|u| \le \Xi/Q} g(u) \sum_{h \sim H}  \e\bigg(\frac{Xh^\alpha u}{H^\alpha M^\beta N^\gamma} \bigg) \bigg|, 
\end{equation}
where $u = u_{\mathbf{m}}(\mathbf{n}) \defeq m^\beta n^\gamma - \tilde{m}^\beta \tilde{n}^\gamma, \mathbf{m} \defeq (m, \tilde{m}), \mathbf{n} \defeq (n, \tilde{n}), g(u) \defeq 1 - Q|u|/\Xi$ and $\Xi \defeq c_1 M^\beta N^\gamma$. 

First we consider the case when $X \le MN$. The right-hand side of $\eqref{eq:afC}$ is 
$$
\ll H Q \sum_{m \sim M} \sum_{\tilde{m} \sim M} \mathop{\sum_{n \sim N} \sum_{\tilde{n} \sim N}}_{|u| \le \Xi/Q} \bigg| \sum_{h \sim H}  \e\bigg(\frac{Xh^\alpha u}{H^\alpha M^\beta N^\gamma} \bigg) \bigg|.
$$
Take $Q \defeq \max\{1,  X/(\eps H)\}$. By a trivial estimation and Lemma \ref{lem:spacing}, the contribution of the case $|u|\leqslant \Xi \cL/(MN)$ is
$$
\ll HQ \, \fA(M,N; \frac{\cL}{MN}) H \ll \cL H^{2}MN + \cL XHMN.
$$
From Lemma~\ref{lem:KL}, the contribution of the case $\Xi \cL/(MN) \leqslant |u| \leqslant \Xi/Q$ is
$$
\ll HQ\cL\max_{\cL /MN \leqslant \Delta \leqslant 1/Q} \fA(M, N; \Delta) (X^{-1}\Delta^{-1}H) \ll \cL X^{-1}H^{2}M^{2}N^{2}+ \cL HM^{2}N^{2}.
$$
Then we have
\begin{equation}
\label{eq:leMN}
|\fT|\ll \cL( HM^{1/2}N^{1/2} + X^{-1/2}HMN + H^{1/2}MN ).
\end{equation}

Let $X \ge MN$. A trivial estimation together with Lemma \ref{lem:spacing} provides that the contribution from $|u| \le \Xi \cL / (MN)$ is $\ll H^2 MNQ\cL$. We apply a dynamic argument for some $\Delta \in (\cL/ (MN), 1/Q]$ to obtain
\begin{equation}
	\label{eq:st}
	|\fT|^2 \ll H Q \sum_{m \sim M} \sum_{\tilde{m} \sim M} \bigg| \mathop{\sum_{n \sim N} \sum_{\tilde{n} \sim N}}_{|u|/\Xi \sim \Delta} g(u) \sum_{h \sim H}  \e\bigg(\frac{Xh^\alpha u}{H^\alpha M^\beta N^\gamma} \bigg) \bigg| + H^2 MNQ \cL. 
\end{equation}

For $H' \defeq X\Delta /H \le \eps$, applying Lemma \ref{lem:KL} and Lemma \ref{lem:spacing}, the first term of the right-hand side of \eqref{eq:st} is $\ll X^{-1} (HMN)^2 Q \cL \ll H^2 MNQ \cL$. 

Now we suppose that $H' \ge \eps$. Applying Lemma \ref{lem:B}, the innermost sum of the first term of the right-hand side of \eqref{eq:st} is 
\begin{equation}
\label{eq:T}
\ll (X\Delta) ^{-1/2} H \sum_{h' \in I'} \e(f(u, h')) + O(\cL + (X\Delta)^{-1/2} H ),
\end{equation}
where
$$
f(u, h') \defeq c_2 X \Delta \big( \frac{u}{\Delta \Xi} \big)^{\frac{1}{1-\alpha}} \big( \frac{h'}{H'} \big)^{\frac{\alpha}{\alpha - 1}}
$$
and $I' \subset [H' , 2H']$. Again applying Lemma \ref{lem:spacing}, we deduce that
\begin{align*}
|\fT|^2 &\ll (X\Delta)^{-1/2} H^2 Q \sum_{h' \sim H'} S(\Delta) + \Delta H Q M^2 N^2 \\
&\qquad + X^{-1/2} \Delta^{1/2} H^2 M^2 N^2 + H^2 M N Q \cL,
\end{align*}
where
$$
S(\Delta) \defeq \sum_{m \sim M} \sum_{\tilde{m} \sim M} \bigg| \mathop{\sum_{n \sim N} \sum_{\tilde{n} \sim N}}_{|u|/\Xi \sim \Delta} \e (f(u, h')) \bigg|. 
$$

The set $\{(m / \tilde{m})^{\beta / \gamma} : m \sim M, \tilde{m} \sim M\}$ is bounded and let us say it is $\subset [c_3, c_4]$. Let $I \defeq [c_3, c_4]$ and 
$$
J \defeq \fl{\frac{|I|}{\eps \Delta} } + 1. 
$$
We split $I$ into $J$ subintervals $I_j$ of length $\delta := |I| / J$. For each $d \in \mathbb{N}$, we put
\[\mathcal{M}_j(d) := \{\mathbf{m} : m \sim M, \tilde{m} \sim M, (m, \tilde{m}) = d, \big( \frac{m}{\tilde{m}} \big)^{\beta / \gamma} \in I_j\},\]
and for every $\mathbf{m} \in \mathcal{M}_j(d)$ we define
\[\mathcal{N}(\mathbf{m}) := \{\mathbf{n} : n \sim N, \tilde{n} \sim N, \Delta < \frac{|u|}{\Xi} \le 2 \Delta\}.\]

Thus we have
\begin{equation}
\label{eq:2.3}
S(\Delta) = \sum_{d \le 2M} \sum_{j \le J} \sum_{\mathbf{m} \in \mathcal{M}_j(d)} |S_j(\Delta, \mathbf{m})|,
\end{equation}
where
$$
S_j(\Delta, \mathbf{m}) := \sum_{\mathbf{n} \in \mathcal{N}(\mathbf{m})} e(f(u)).
$$
Here $f(u) \defeq Y u^{\frac{1}{1-\alpha}}$ with 
$$
Y \defeq c_2 X \Delta \big( \frac{1}{\Delta \Xi} \big)^{\frac{1}{1-\alpha}} \big( \frac{h'}{H'} \big)^{\frac{\alpha}{\alpha - 1}}.
$$
We write $f(u)$ instead of $f(u,h')$ for short, also because we are not discussing $h'$ at this stage. 

For $d_j \in \mathbb{Z}$, we define $\Gamma_{d_j} := \{(x, y) \in \mathbb{R}^2 : y = (a_j x - d_j)/b_j\}$. Clearly when $d_j$ runs over $\mathbb{Z}$, $\mathbb{Z}^2 \cap \Gamma_{d_j}$ constitutes a partition of $\mathbb{Z}^2$. Hence we can write
\begin{equation}
\label{eq:2.4}
S_j(\Delta, \mathbf{m}) \ll \sum_{d_j \in \mathbb{Z}} \left| \sum_{\mathbf{n} \in \mathcal{N}(\mathbf{m}) \cap \Gamma_{d_j}}  e(f(u)) \right|.
\end{equation}

Next we cite a lemma, which is Lemma 2.3 by Wu \cite{Wu2002}. 

\begin{lemma}
\label{lem:2.3}
For $\mathcal{N}(\mathbf{m}) \cap \Gamma_{d_j}$, we have the following two properties:
\begin{enumerate}
	\item[(i)] $\mathcal{N}(\mathbf{m}) \cap \Gamma_{d_j}$ is two segments at most.
	\item[(ii)] If $\mathcal{N}(\mathbf{m}) \cap \Gamma_{d_j} \neq \emptyset$ and if $\delta \le \varepsilon \Delta$, then $d_j \asymp \Delta N b_j$.
\end{enumerate}
\end{lemma}

Let $(n_0, \tilde{n}_0)$ be the starting point of $\mathcal{N}(\mathbf{m}) \cap \Gamma_{d_j} \neq \emptyset$. Then $b_j \tilde{n}_0 = a_j n_0 - d_j$. Now we write the parametric equation of this straight line as follows:
\[
n = b_j l + n_0, \quad \tilde{n} = a_j l + \tilde{n}_0
\]
with
$$
0 \le l \le c_5 \frac{N}{b_j} \asymp \frac{N}{a_j}. 
$$
Put 
$$
u(t) := m^\beta (b_j t + n_0)^\gamma - \tilde{m}^\beta (a_j t + \tilde{n}_0)^\gamma
$$
and $F(t) := f(u(t))$. From \eqref{eq:2.4} and Lemma \ref{lem:2.3}, we deduce
\begin{align}
	\label{eq:2.5}
	S_j(\Delta, \mathbf{m}) &\ll 
	\begin{cases}
		\Delta N b_j \bigg| \displaystyle \sum_{l \le c_5 N / b_j} e(F(l)) \bigg| & \text{if } b_j \le N, \\
		\Delta N b_j & \text{if } b_j > N,
	\end{cases}  
\end{align}

We need the behavior of $F(l)$ to estimate the innermost sum, which is described by the following lemma (see \cite[Lemma 2.4]{Wu2002}). 

\begin{lemma}
	\label{lem:2.4}
	If $\delta \le \varepsilon \Delta$ and $(\gamma - 1)/(1 - \alpha) \neq 0, 1, \ldots, k - 1$, then
	\begin{align*}
		u(t) &= A(b_j t + n_0)^{\gamma-1} \{1 + O(\varepsilon + \Delta)\}, \\
		F^{(i)}(t) &= B(b_j t + n_0)^{\frac{\gamma - 1 } {1 - \alpha} -i} b_j^i \{1 + O_k(\varepsilon + \Delta)\} \quad (0 \le i \le k), 
	\end{align*}
	where
	\begin{align*}
		A &:= \gamma \tilde{m}^\beta ( \frac{a_j}{b_j})^{\gamma - 1} (\frac{d_j}{b_j}) \asymp \Delta M^\beta N, \\
		B &:= \prod_{i=0}^{k-1} \big( \frac{\gamma - 1 } {1 - \alpha} - i \big) YA^{ \frac{1}{1-\alpha}} \asymp X \Delta N^{-\frac{ \gamma - 1 } {1 - \alpha}}.
	\end{align*}
\end{lemma}

Now we apply the exponent pair method (Lemma \ref{lem: exponent pair method}) with 
$$
(Y, N) \to (X\Delta\frac{b_j}{N}, \frac{N}{b_j}). 
$$
We obtain that
\begin{equation*}
\label{eq:imp}
\sum_{l \le c_5 N/b_j} \e(F(l)) \ll (X\Delta)^{\kappa} \big(\frac{N}{b_j}\big)^{\lambda - \kappa} + (X \Delta)^{-1} \frac{N}{b_j}. 
\end{equation*}
Combine with equation \eqref{eq:2.5}, we obtain that
$$
S_j(\Delta, \mathbf{m}) \ll \Delta N^2 \cL \Theta_j
$$
with
\begin{align*}
	\Theta_j &\ll 
	\begin{cases}
		\min \{1, \Theta(b_j)\} & \text{if } b_j \le N, \\
		b_j/N & \text{if } b_j > N,
	\end{cases}  
\end{align*}
where
$$
\Theta(b_j) \defeq (X\Delta)^{\kappa} \big( \frac{N}{b_j} \big)^{\lambda - \kappa - 1} + (X\Delta)^{-1}. 
$$
Inserting into equation \eqref{eq:2.3} yields
\begin{equation}
\label{eq:2.9}
S(\Delta) \ll \Delta N^2 \cL \sum_{d \le 2M} \sum_{\substack{j \le J \\ b_j \le N}} \sum_{\mathbf{m} \in \mathcal{M}_j(d)} \Theta_j + \Delta N^2 \cL \sum_{d \le 2M} \sum_{\substack{j \le J \\ b_j > N}} \sum_{\mathbf{m} \in \mathcal{M}_j(d)} \Theta_j. 
\end{equation}

Next we apply a Lemma due to Huxley and Watt (see Lemma 1.6.1 in \cite{H1996}). 

\begin{lemma}
	\label{lem:2.2}
	We pick $a_j / b_j \in I_j \cap \mathbb{Q}$ with $b_j$ least and subject to $b_j \ge \delta^{-1/2}$. Then for any $B > 0$, we have $|\{1 \le j \le J : b_j \ge B\}| \ll 1 / (\delta B)^2$. In particular $b_j \ll \delta^{-1}$ for $1 \le j \le J$.
\end{lemma}

By using Lemma \ref{lem:2.2}, we have
$$
\sum_{\substack{j \le J \\ b_j > N}} \Theta_j \ll \sum_{0 \le i \ll \cL} \sum_{\substack{j \le J \\ 1 \le b_j / 2^iN \le 2}} \frac{b_j}{N} \ll \sum_{0 \le i \ll \cL} 2^{-i}(N\delta)^{-2} \ll (N\delta)^{-2}.
$$
Thus the second member on the right-hand side of \eqref{eq:2.9} is
\begin{equation}
\label{eq:2.10}
\ll \Delta N^2 (\delta M^2 + M)(N\delta)^{-2} \cL \ll (M^2 + \Delta^{-1} M) \cL. 
\end{equation}
Similarly, Lemma \ref{lem:2.2} allows us to deduce
$$
\sum_{\substack{j \le J \\ b_j \le N}} \Theta_j \ll \sum_{0 \le i \ll \cL} \sum_{\substack{j \le J \\ 1 \le b_j / 2^i \delta^{-1/2} \le 2}} \Theta(b_j) \ll \sum_{i \geq 0} 4^{-i} \delta^{-1} \Theta(2^i \delta^{-1/2}) \ll \delta^{-1} \Theta(\delta^{-1/2}).
$$
Thus the first member on the right-hand side of \eqref{eq:2.9} is
\begin{align}
	\label{eq:2.11}
	&\ll \Delta N^2 \cL \sum_{d \le \sqrt{\delta}M} \sum_{\substack{j \le J \\ b_j \le N}} \sum_{\mathbf{m} \in \mathcal{M}_j(d)} \Theta(b_j) + \Delta N^2 \cL \sum_{d > \sqrt{\delta}M} \sum_{\substack{j \le J \\ b_j \le N}} \sum_{\mathbf{m} \in \mathcal{M}_j(d)} 1  \nonumber \\
	&\ll \Delta N^2 \cL \sum_{d \le \sqrt{\delta}M} \delta(M/d)^2 \sum_{\substack{j \le J \\ b_j \le N}} \Theta(b_j) + \Delta N^2 \cL \sum_{d > \sqrt{\delta}M} (M/d)^2 \nonumber \\
	&\ll \Delta (MN)^2 \Theta(\delta^{-1/2}) \cL + \Delta \delta^{-1/2} MN^2 \cL \nonumber \\
	&\ll \Delta (MN)^2 \Theta(\Delta^{-1/2}) \cL + \Delta^{1/2} MN^2 \cL.   
\end{align}
Inserting $\eqref{eq:2.10}$ and $\eqref{eq:2.11}$ to $\eqref{eq:2.9}$, we obtain
\begin{equation}
\label{eq:2.12}
S(\Delta) \cL^{-1} \ll X^\kappa \Delta^{(1 + \kappa + \lambda)/2} M^2 N^{1 + \lambda - \kappa} + X^{-1} M^2 N^2 + \Delta^{1/2} M N^2 + M^2 + \Delta^{-1} M. 
\end{equation}
Inserting $\eqref{eq:2.12}$ into $\eqref{eq:T}$ and by $\cL/(MN) \le \Delta \le 1/Q$ we obtain that
\begin{align*}
|\fT|^2 \cL^{-2} &\ll X^{1/2 + \kappa} Q^{-\frac{1}{2}(\kappa + \lambda)} H M^2 N^{1+\lambda - \kappa} + X^{1/2} Q^{1/2} H M^2  \\
&+ X^{-1/2} H Q^{1/2} M^2 N^2 + X^{1/2} Q H M^{3/2} N^{1/2} + H^2 M N Q \\
&+ X^{1/2} H M N^2 + H M^2 N^2 + X^{-1/2} Q^{-1/2} H^2 M^2 N^2. 
\end{align*}

Let $Q \in [(MN)^{1/3}, MN]$. Then $X^{-1/2} Q^{-1/2} H^2 M^2 N^2 \ll H^2 MNQ$. Hence by Lemma $\ref{lem:7}$ we have
\begin{align*}
|\fT|^2\cL^{-2} &\ll X^{\frac{1+2\kappa}{2 + \kappa + \lambda}} H^{\frac{2(\kappa + \lambda + 1)}{2+\kappa + \lambda}} M^{\frac{\kappa+\lambda + 4}{2+\kappa + \lambda}} N^{\frac{2-\kappa+3\lambda}{2+\kappa+\lambda}} + X^{\frac{3\kappa + \lambda + 1}{2(1+\kappa + \lambda)}} H M^2 N^{\frac{1 + \lambda - \kappa}{1 + \kappa + \lambda}} \\
&\qquad + X^{\frac{\kappa-\lambda +1}{2(1+\kappa+\lambda)}} H M^2 N^{\frac{\kappa + 3\lambda + 1}{1+\kappa + \lambda}} + X^{\frac{5\kappa+\lambda + 2}{2(2+\kappa+\lambda)}} H M^{\frac{3\kappa+3\lambda + 8}{2(2+\kappa+\lambda)}} N^{\frac{4+5\lambda-3\kappa}{2(2+\kappa+\lambda)}} \\
&\qquad + X^{\frac{1}{2}}HM^{\frac{13}{6}} N^{\frac{1}{6}} + H^2 M^{\frac{4}{3}} N^{\frac{4}{3}} + X^{-\frac{1}{2}}H M^{\frac{13}{6}} N^{\frac{13}{6}} \\
&\qquad + X^{\frac{1}{2}}HM^{\frac{11}{6}} N^{\frac{5}{6}} + X^{\frac{1}{2}+\kappa} H M^{2-\frac{1}{2}(\kappa+\lambda)} N^{1+ \frac{1}{2}\lambda - \frac{3}{2}\kappa} \\
&\qquad + X^{1/2}HMN^{2} + HM^2N^2.
\end{align*}
Combine with the result when $X \le MN$ in \eqref{eq:leMN}, we prove the bound of $\fT$.

\section{Proof of Theorem~\ref{thm:main}}
	
\subsection{Preparations}
	
We note that the details of set-up of the Piatetski-Shapiro prime number theorem are well written in \cite[P. 46--49]{GraKol}, \cite[P. 245--247]{HB2}. However, we aim to write a ``self-contained" article, so a few key steps are still written here.
	
Reminding that $\gamma \defeq c^{-1}$, it follows that $p = \fl{n^c}$ if and only if $-(p+1)^\gamma < -n \le -p^\gamma$, which provides that
$$
\pi_c(x) = \sum_{p \le x} \big( \fl{-p^\gamma} - \fl{-(p+1)^\gamma} \big) + O(1). 
$$
Recall that $\psi(t) \defeq t - \fl{t} - 1/2$. We conclude that
$$
\pi_c(x) = \sum_{p \le x} \big( (p+1)^\gamma-p^\gamma \big) + \sum_{p \le x} \big( \psi(-p^\gamma) - \psi(-(p+1)^\gamma) \big) + O(1).
$$
By a simple partial summation and the prime number theorem, we achieve that
$$
\sum_{p \le x} \big( (p+1)^\gamma-p^\gamma \big) = \frac{x^\gamma}{\log x} + O\big(\frac{x^\gamma}{\log^2 x}\big). 
$$
Therefore it is sufficient to prove that
\begin{equation}\label{eq:base}
    \sum_{n \sim x/2} \Lambda(n)  \big( \psi(-n^\gamma) - \psi(-(n+1)^\gamma) \big) \ll x^{\gamma-\eps}.
\end{equation}
Next we apply Lemma \ref{lem:Vaaler}, which shows that the left-hand side of \eqref{eq:base} is equal to $S + O(S')$, where
$$
S \defeq \sum_{0 < h \le H} \sum_{n \sim x/2} \Lambda(n) a_h \big( \e(hn^\gamma) - \e(h(n+1)^\gamma) \big)
$$
and 
$$
S' \defeq \sum_{0 \le h \le H} \sum_{n \sim x/2} \Lambda(n) b_h \big( \e(hn^\gamma) + \e(h(n+1)^\gamma) \big),
$$
with $H \defeq x^{1-\gamma+\eps}$. We only deal with the bound of $S$ since the bound of $S'$ when $h\neq 0$ can be estimated by the same way of $S$ and the bound of $S'$ when $h = 0$ is trivial with the choice of $H$. 
Now let 
$$
\phi_h(x) \defeq 1 - \e( h(x+1)^\gamma - hx^\gamma ). 
$$
We have by partial summation that
\begin{align*}
	S &\ll \sum_{0 < h \le H} h^{-1} |\phi_h(x_1)| \bigg| \sum_{x/2 \le n \le x_1} \Lambda(n) \e(hn^\gamma) \bigg| \\
	&\qquad + \int_{x/2}^{x_1} \sum_{0 < h \le H} h^{-1} \bigg|\frac{\partial \phi_h(t)}{\partial t}\bigg| \bigg| \sum_{x/2 \le n \le t} \Lambda(n) \e(hn^\gamma) \bigg| dt \\
	&\ll x^{\gamma-1} \max_{x/2 < x_2 \le x} \sum_{0 < h \le H} \bigg| \sum_{x/2 \le n \le x_2} \Lambda(n) \e(hn^\gamma) \bigg|,
\end{align*}
where $x/2 \le x_2 \le x_1 \le x$ together with the bounds
$$
\phi_h(t) \ll h t^{\gamma-1} \quad \text{and} \quad \frac{\partial\phi_h(t)}{\partial t} \ll ht^{\gamma-2}. 
$$
Applying a dynamic summation argument, we conclude that it is sufficient to prove that 
\begin{equation*}\label{eq:main}
	\sum_{h \sim H} \delta_h \sum_{n \sim x} \Lambda(n) \e(hn^\gamma) \ll x^{1-\eps}. 
\end{equation*}

\subsection{Applying the bounds of exponential sums}
Here we only consider the case when $\gamma < 13/15$, since Liu and Rivat \cite{LiRi} proved the Piatetski-Shapiro prime number theorem for $13/15 < \gamma < 1$. We apply Lemma \ref{lem:RS} with the choice $l=5$ and $\alpha \defeq 1-\gamma$. It is sufficient to prove that
\begin{enumerate}
	\item[(i)] type $II$: $S_{II} \ll x^{1-\eps} $ for $x^{\alpha} \ll N \ll x^{2\alpha}$;
	\item [(ii)] type $I'$: $S_{I} \ll x^{1-\eps}$ for $x^{2\alpha} \ll N \ll x^{1/3}$;
	\item[(iii)] type $I$: $S_{I} \ll x^{1-\eps}$ for $x^{(1-\alpha)/2} \ll N$
	\end{enumerate}
where $S_{I}$ and $S_{II}$ are defined as \eqref{eq:SI} and \eqref{eq:SII}. By Lemma \ref{lem:HBTII} with 
$$
2(1-\gamma) < 5\gamma-4,
$$
it follows that $S_{II} \ll x^{1-\eps}$ for $\gamma > 6/7$. 

\subsection{Type $I$ sum}
We have the following bound for Type $I$ estimation. 
\begin{lemma}
	\label{cor:ty1}
	With the same notations as \eqref{eq:S}, if
	$$
	M \leqslant x^{3\gamma -2-\eps},
	$$
	we have $	S_{I} \ll x^{1-\eps}$.
\end{lemma}
\begin{proof}
	For $S_{I}$, we trivially have
	\begin{align*}
		S_{I}&\leqslant 	\sum_{h\sim H}\sum_{m\sim M}|a_{m}|\bigg|\sum_{n\sim N}b_n\e(h(mn)^{\gamma})\bigg|\\
		&\ll \sum_{h \sim H}  \sum_{m \sim M} N^{\eps}\max_{N \le N_1 \le N_2 \le 2N} \bigg| \sum_{N_1 \le n \le N_2} \e (h(mn^{\gamma}))\bigg|.
	\end{align*}
	Use Lemma~\ref{lem:ty} to above summation with
	$$
	(X,H,M,N,\alpha,\beta,\gamma) \to (HM^{\gamma}N^{\gamma},H,N,M,\gamma,1,\gamma)
	$$
	and we deduce
	$$
	S_{I}\ll (HMN)^{\eps}(HM^{\frac{\gamma+3}{4}}N^{\frac{\gamma+2}{4}} + HMN^{\frac{1}{2}} + M^{1-\gamma}N^{1-\gamma}   ).
	$$
	For
	$$
	MN \asymp x \quad \text{and} \quad H\leqslant x^{1-\gamma+\eps},
	$$
	it follows that
	$$
	S_{I}\ll x^{\eps}( x^{-\frac{3}{4}\gamma + \frac{3}{2}}M^{\frac{1}{4}} + x^{-\gamma + \frac{3}{2}}M^{\frac{1}{2}} + x^{1-\gamma}).
	$$
	Let the upper bound of $S_{I}\ll x^{1-\eps}$ and we finish the proof.
\end{proof}

By Lemma \ref{cor:ty1}, with 
$$
1-\frac{\gamma}{2} < 3\gamma -2  
$$
it follows that $S_{I} \ll x^{1-\eps}$ with $\gamma > 6/7$. 

\subsection{Type $I'$ sum}
Now it is sufficient to prove that $S_{I} \ll x^{1-\eps}$ with $x^{5\gamma-4} \ll N \ll x^{1/3}$ or $x^{2/3} \ll M \ll x^{5-5\gamma}$. We apply Theorem \ref{thm:2} with
$$
(X, H, M, N) \to (x, M, H, N)
$$
and the exponent pair (see Theorem 20 in \cite{TTY2025})
$$
(\kappa, \lambda) = \bigg(\frac{10769}{351096}, \frac{609317}{702192} \bigg). 
$$
Therefore we obtain that
$$
S_{I} \cL^{-1} \ll E_1 + E_2 + \cdots + E_9,
$$
where the worst term 
$$
E_1 = x^{\frac{3697844}{2035239}-\frac{3439623}{4070478} \gamma} M^{-\frac{544703}{4070478}}. 
$$
With $E_1 \ll x^{1-\eps}$ it requires that 
$$
M \gg x^{\frac{3325210-3439623\gamma}{544703} + \varepsilon}. 
$$
We require that $M \gg x^{2/3}$, which follows that 
\begin{equation}\label{eq:gammarestriction}
\gamma > \frac{8886224}{10318869} = 0.86116 \dots. 
\end{equation}
The rest terms are also listed as followings. 
\begin{align*}
E_2 &\ll x^{1.1743-\gamma} M^{0.016} \\
E_3 &\ll x^{1.98-\gamma} M^{-0.458} \\
E_4 &\ll x^{1.896-0.923\gamma} M^{-0.212} \\
E_5 &\ll x^{\frac{17}{12} - \frac{13}{12}\gamma} M^{\frac{5}{12}} \\
E_6 &\ll x^{\frac{4}{3} - \frac{2}{3}\gamma} M^{\frac{1}{3}}\\
E_7 &\ll x^{\frac{23}{12} - \frac{13}{12} \gamma} M^{-\frac{7}{12}} \\
E_8 &\ll x^{\frac{19}{12} - \frac{11}{12}\gamma} M^{\frac{1}{12}} \\
E_9 &\ll x^{1.736 - 0.776\gamma} M^{-0.194}. 
\end{align*}
With $E_i \ll x^{1-\eps}$ for $2 \le i \le 9$, we require that
\begin{align*}
	\left\{\begin{array}{ll}
		M \ll x^{62.5\gamma - 46.438 -\varepsilon} & \\
		M \gg x^{2.14-2.183\gamma + \varepsilon}, & \\
		M \ll x^{4.227-4.353\gamma - \varepsilon}, & \\
		M \ll x^{\frac{13}{5}\gamma - 1 - \varepsilon}, & \\
		M \ll x^{2\gamma - 1 - \varepsilon}, & \\
		M \gg x^{\frac{11}{7} - \frac{13}{7}\gamma + \varepsilon}, & \\
		M \ll x^{11\gamma-7 - \varepsilon}, & \\
		M \gg x^{3.794 - 4\gamma + \varepsilon}. 
	\end{array}\right.
\end{align*}
To satisfy that $x^{2/3} \ll M \ll x^{5 - 5\gamma}$, we obtain that
\begin{align*}
	\left\{\begin{array}{ll}
		\gamma > 0.762 & \\
		\gamma > 0.675, & \\
		\gamma > 0.838, & \\
		\gamma > \frac{15}{19}, & \\
		\gamma > \frac{6}{7}, & \\
		\gamma > \frac{19}{39}, & \\
		\gamma > \frac{3}{4}, & \\
		\gamma > 0.782. 
	\end{array}\right.
\end{align*}
This shows that \eqref{eq:gammarestriction} provides the restriction of $\gamma$. 

\section{Acknowledgement}
This work was supported by the National Natural Science Foundation of China (No. 11901447, 12271422), the Natural Science Foundation of Shaanxi Province (No. 2024JC-YBMS-029).

\end{document}